\documentclass[12pt]{amsart}
\usepackage{amsmath,amsthm}
\usepackage{amssymb}
\usepackage[all]{xypic}
\usepackage{euscript}
\begin{document} 
%
%
\theoremstyle{plain}
\swapnumbers
	\newtheorem{thm}{Theorem}[section]
	\newtheorem{prop}[thm]{Proposition}
	\newtheorem{lemma}[thm]{Lemma}
	\newtheorem{cor}[thm]{Corollary}
	\newtheorem{fact}[thm]{Fact}
	\newtheorem{subsec}[thm]{}
\theoremstyle{definition}
	\newtheorem{assume}[thm]{Assumption}
	\newtheorem{defn}[thm]{Definition}
	\newtheorem{example}[thm]{Example}
	\newtheorem{examples}[thm]{Examples}
	\newtheorem{claim}[thm]{Claim}
	\newtheorem{notn}[thm]{Notation}
	\newtheorem{construct}[thm]{Construction}
\theoremstyle{remark}
        \newtheorem{remark}[thm]{Remark}
        \newtheorem{remarks}[thm]{Remarks}
	\newtheorem{ack}[thm]{Acknowledgements}
\newenvironment{myeq}[1][]
{\stepcounter{thm}\begin{equation}\tag{\thethm}{#1}}
{\end{equation}}
\newcommand{\mydiag}[2][]{\myeq[#1]{\xymatrix{#2}}}
\newcommand{\mydiagram}[2][]
{\stepcounter{thm}\begin{equation}
     \tag{\thethm}{#1}\vcenter{\xymatrix{#2}}\end{equation}}
\newcommand{\mydiagr}[2][]
{\stepcounter{thm}\begin{equation}
     \tag{\thethm}{#1}\vcenter{\xymatrix@R=15pt{#2}}\end{equation}}
%
\newcommand{\mypict}[2][]
{\stepcounter{thm}\begin{equation}
     \tag{\thethm}{#1}\vcenter{\begin{picture}{#2}\end{picture}}\end{equation}}
\newenvironment{mysubsection}[2][]
{\begin{subsec}\begin{upshape}\begin{bfseries}{#2.}
\end{bfseries}{#1}}
{\end{upshape}\end{subsec}}
\newenvironment{mysubsect}[2][]
{\begin{subsec}\begin{upshape}\begin{bfseries}{#2\vsm.}
\end{bfseries}{#1}}
{\end{upshape}\end{subsec}}
\newcommand{\sect}{\setcounter{thm}{0}\section}
\newcommand{\wh}{\ -- \ }
\newcommand{\w}[2][ ]{\ \ensuremath{#2}{#1}\ }
\newcommand{\ww}[1]{\ \ensuremath{#1}}
\newcommand{\wb}[2][ ]{\ (\ensuremath{#2}){#1}\ }
\newcommand{\wref}[2][ ]{\ \eqref{#2}{#1}\ }
%
%
\newcommand{\xra}[1]{\xrightarrow{#1}}
\newcommand{\xla}[1]{\xleftarrow{#1}}
\newcommand{\lar}{\leftarrow}
\newcommand{\Ra}{\Rightarrow}
\newcommand{\hra}{\hookrightarrow}
\newcommand{\lora}{\longrightarrow}
\newcommand{\bstar}{\mbox{\large $\star$}}
\newcommand{\adj}[2]{\substack{{#1}\\ \rightleftharpoons \\ {#2}}}
\newcommand{\hsq}{\hspace*{13 mm}}
\newcommand{\hsp}{\hspace*{10 mm}}
\newcommand{\hs}{\hspace*{5 mm}}
\newcommand{\hsm}{\hspace*{2 mm}}
\newcommand{\hsn}{\hspace*{-20 mm}}
\newcommand{\hsnn}{\hspace*{-27 mm}}
\newcommand{\vsn}{\vspace{-30 mm}}
\newcommand{\vsp}{\vspace{10 mm}}
\newcommand{\vs}{\vspace{7 mm}}
\newcommand{\vsm}{\vspace{2 mm}}
\newcommand{\rest}[1]{\lvert_{#1}}
\newcommand{\lra}[1]{\langle{#1}\rangle}
\newcommand{\EQUIV}{\Leftrightarrow}
\newcommand{\epic}{\to\hspace{-5 mm}\to}
\newcommand{\xepic}[1]{\xrightarrow{#1}\hspace{-5 mm}\to}
\newcommand{\hotimes}{\hat{\otimes}}
\newcommand{\hy}[2]{{#1}\text{-}{#2}}
\newcommand{\vare}{\varepsilon}
\newcommand{\vart}{\vartheta}
\newcommand{\hvart}{\hat{\vart}}
%
%
\newcommand{\ab}{\operatorname{ab}}
\newcommand{\Aut}{\operatorname{Aut}}
\newcommand{\BW}{\operatorname{BW}}
\newcommand{\Cok}{\operatorname{Coker}\,}
\newcommand{\colim}{\operatorname{colim}}
\newcommand{\Cotor}{\operatorname{Cotor}}
\newcommand{\Ext}{\operatorname{Ext}}
\newcommand{\gr}{\operatorname{gr}}
\newcommand{\ho}{\operatorname{ho}}
\newcommand{\Hom}{\operatorname{Hom}}
\newcommand{\Id}{\operatorname{Id}}
\newcommand{\tId}{\widetilde{Id}}
\newcommand{\Image}{\operatorname{Im}\,}
\newcommand{\inc}{\operatorname{inc}}
\newcommand{\Ker}{\operatorname{Ker}\,}
\newcommand{\map}{\operatorname{map}}
\newcommand{\mapa}{\map_{\ast}}
\newcommand{\mape}[2]{\map\sp{#1}\sb{#2}}
\newcommand{\Obj}{\operatorname{Obj}}
\newcommand{\op}{\sp{\operatorname{op}}}
\newcommand{\real}{\operatorname{real}}
\newcommand{\res}{\operatorname{res}}
\newcommand{\sk}[1]{\operatorname{sk}_{#1}}
\newcommand{\Tor}{\operatorname{Tor}}
\newcommand{\Tot}{\operatorname{Tot}}
\newcommand{\we}{\operatorname{w.e.}}
\newcommand{\sss}{\hspace*{1 mm}\sp{s}}
\newcommand{\vvv}{\hspace*{1 mm}\sp{v}}
%
%
\newcommand{\A}{{\EuScript A}}
\newcommand{\hA}{\hat{\A}}
\newcommand{\tA}{\tilde{\A}}
\newcommand{\cA}{{\mathcal A}}
\newcommand{\Ap}{\cA_{p}}
\newcommand{\B}{{\EuScript B}}
\newcommand{\cB}{{\mathcal B}}
\newcommand{\hB}{\hat{\B}}
\newcommand{\AB}{(\A,\B)}
\newcommand{\tB}{\tilde{\B}}
\newcommand{\C}{{\mathcal C}}
\newcommand{\hC}{\hat{\C}}
\newcommand{\tC}{\tilde{\C}}
\newcommand{\Tow}[1]{\hspace*{-0.5mm}\sp{\#}{#1}}
\newcommand{\Ct}{\Tow{\C}}
\newcommand{\Ctl}[1]{\hspace*{-1mm}\sb{{#1}\leq}\sp{\hspace*{1.7mm}\#}
                 \hspace{0.5mm}\C}
\newcommand{\Ctg}[1]{\hspace*{-1mm}\sb{{#1}\geq}\sp{\hspace*{1.7mm}\#}
                  \hspace{0.5mm}\C}
\newcommand{\D}{{\mathcal D}}
\newcommand{\Dpp}{\D\sp{+}}
\newcommand{\hD}{\hat{\D}}
\newcommand{\E}{{\mathcal E}}
\newcommand{\tE}{\tilde{\E}}
\newcommand{\F}{{\mathcal F}}
\newcommand{\tF}{\tilde{F}}
\newcommand{\FA}{\F_{\A}}
\newcommand{\G}{{\mathcal G}}
\newcommand{\HH}{{\mathcal H}}
\newcommand{\K}{{\mathcal K}}
\newcommand{\LL}{{\mathcal L}}
\newcommand{\M}{{\mathcal M}}
\newcommand{\N}{{\mathcal N}}
\newcommand{\OO}{{\mathcal O}}
\newcommand{\Op}{\OO\sp{+}}
\newcommand{\Om}{\OO\sp{-}}
\newcommand{\PP}{{\mathcal P}}
\newcommand{\bP}[1]{\overline{\PP}[{#1}]}
\newcommand{\QQ}{{\mathcal Q}}
\newcommand{\eS}[1]{{\EuScript S}\lra{#1}}
\newcommand{\Stem}{{\EuScript Stem}}
\newcommand{\Pstem}{{\EuScript P\!stem}}
\newcommand{\Ss}{{\mathcal S}}
\newcommand{\U}{{\mathcal U}}
\newcommand{\V}{{\mathcal V}}
\newcommand{\X}{{\mathcal X}}
%
%
\newcommand{\uA}[1]{\hspace*{0.5mm}\sp{#1}\hspace{-0.3mm}A}
\newcommand{\up}[1]{\hspace*{0.5mm}\sp{#1}\hspace{-0.3mm}p}
\newcommand{\sq}[1]{\hspace*{0.5mm}\sb{#1}\hspace{-0.3mm}q}
\newcommand{\suq}[2]{\hspace*{0.5mm}\sb{#1}\sp{#2}\hspace{-0.3mm}q}
\newcommand{\tk}[2]{\tau_{#1}[{#2}]}
\newcommand{\Ok}[3]{{#1}_{#2}[{#3}]}
\newcommand{\wk}[2]{{#1}\lra{#2}}
\newcommand{\uOc}[3]{\hspace*{0.5mm}\sp{#2}\hspace{-0.3mm}\wk{{#1}}{#3}}
\newcommand{\sOk}[3]{\hspace*{0.3mm}\sb{#2}\hspace{0.2mm}{#1}[{#3}]}
\newcommand{\suOk}[4]{\hspace*{0.5mm}\sb{#2}\sp{#3}\hspace{-0.3mm}{#1}[{#4}]}
\newcommand{\Ak}[2]{\Ok{A}{#1}{#2}}
\newcommand{\uAc}[2]{\uOc{A}{#1}{#2}}
\newcommand{\sAk}[2]{\sOk{A}{#1}{#2}}
\newcommand{\suAk}[3]{\suOk{A}{#1}{#2}{#3}}
%
%
\newcommand{\Po}[1]{P\sp{#1}}
\newcommand{\uP}[1]{\hspace*{0.5mm}\sp{#1}\hspace{-0.3mm}\PP}
\newcommand{\Pk}[2]{\PP\sb{#1}[{#2}]}
\newcommand{\Pnk}[3]{\Po{#1}\wk{#2}{#3}}
\newcommand{\Pdnk}[2]{\Pd\sp{#1}\lra{#2}}
\newcommand{\Qnk}[2]{Q\sp{#1}\sb{#2}}
\newcommand{\Qk}[1]{Q\sb{#1}}
\newcommand{\Qdnk}[2]{\Qd\sp{#1}\lra{#2}}

\newcommand{\Xk}[1]{\sb{#1}X}
\newcommand{\PPk}[3]{(P[{#2}]{#3})\sb{#1}}
%
%
\newcommand{\Alg}[1]{{#1}\text{-}{\EuScript Alg}}
\newcommand{\tAlg}[1]{{#1}\text{-}{\widetilde{\EuScript Alg}}}
\newcommand{\Ab}{{\EuScript Ab}}
\newcommand{\Abgp}{{\Ab\Gp}}
\newcommand{\sC}{\sss\C}
\newcommand{\vC}{\vvv\C}
\newcommand{\CA}{\C\sb{\A}}
\newcommand{\tCA}{\tilde{\C}\sb{\A}}
\newcommand{\sCA}{\sp{s}\CA}
\newcommand{\hCA}{\sss\tCA}
\newcommand{\vCA}{\vvv\tCA}
\newcommand{\CuA}{\C\sp{\A}}
\newcommand{\tCuA}{\tilde{\C}\sp{\A}}
\newcommand{\vCuA}{\vvv\tCuA}
\newcommand{\sCuA}{\sp{s}\CuA}
\newcommand{\hCuA}{\sss\tCuA}
\newcommand{\Cat}{{\EuScript Cat}}
\newcommand{\Ch}{{\EuScript Chain}}
\newcommand{\DB}{\D_{\B}}
\newcommand{\tDB}{\tilde{\D}_{\B}}
\newcommand{\sDB}{\sp{s}\DB}
\newcommand{\hDB}{\sss\tDB}
\newcommand{\Fs}{\F_{s}}
\newcommand{\Fn}{\{K(\Fp,n)\}_{n=1}^{\infty}}
\newcommand{\Gp}{{\EuScript Gp}}
\newcommand{\Gpd}{{\EuScript Gpd}}
\newcommand{\Hopf}{{\EuScript Hopf}}
\newcommand{\MA}{\M\sb{\A}}
\newcommand{\MVA}{\M\sb{\A}\sp{\V}}
\newcommand{\MAC}{\MA\sp{\tCA}}
\newcommand{\MuA}{\M\sp{\A}}
\newcommand{\MuAC}{\MuA\sb{\tCuA}}
\newcommand{\MuVA}{\MuA\sb{\V}}
\newcommand{\MB}{\M\sb{\B}}
\newcommand{\MBD}{\MB\sp{\tDB}}
\newcommand{\R}[1]{{\mathcal R}_{#1}}
\newcommand{\RM}[1]{{#1}\text{-}{\EuScript Mod}}
\newcommand{\Sn}{\{S^{n}\}_{n=1}^{\infty}}
\newcommand{\Snz}{\{S^{n}\}_{n=0}^{\infty}}
\newcommand{\Sa}{\Ss_{\ast}}
\newcommand{\Set}{{\EuScript Set}}
\newcommand{\Seta}{\Set_{\ast}}
\newcommand{\Track}{{\EuScript Trk}}
\newcommand{\TT}{{\mathcal T}}
\newcommand{\TM}{{\EuScript TM}}
\newcommand{\Ta}{\TT_{\ast}}
\newcommand{\Tat}{\hspace*{-0.5mm}\sp{\#}\Ta}
\newcommand{\VA}{\V\sb{\A}}
%
%
\newcommand{\ihD}[1]{({#1},\hD)}
\newcommand{\iO}[1]{({#1},\OO)}
\newcommand{\iOm}[1]{({#1},\Om)}
\newcommand{\iOp}[1]{({#1},\Op)}
\newcommand{\DO}{\iO{\D}}
\newcommand{\SO}{\iO{\Ss}}
\newcommand{\SD}{(\S,\D)}
\newcommand{\SaO}{\iO{\Sa}}
\newcommand{\SaOm}{\iOm{\Sa}}
\newcommand{\SaOp}{\iOp{\Sa}}
\newcommand{\OC}{\hy{\OO}{\Cat}}
\newcommand{\OmC}{\hy{\Om}{\Cat}}
\newcommand{\OpC}{\hy{\Op}{\Cat}}
\newcommand{\iOC}[1]{\hy{\iO{#1}}{\Cat}}
\newcommand{\iOmC}[1]{\hy{\iOm{#1}}{\Cat}}
\newcommand{\iOpC}[1]{\hy{\iOp{#1}}{\Cat}}
\newcommand{\DOC}{\iOC{\D}}
\newcommand{\SOC}{\iOC{\Ss}}
\newcommand{\SaOC}{\iOC{\Sa}}
\newcommand{\SaOmC}{\iOmC{\Sa}}
\newcommand{\SaOpC}{\iOpC{\Sa}}
\newcommand{\VO}{\iO{\V}}
\newcommand{\VOm}{\iOm{\V}}
\newcommand{\VOp}{\iOp{\V}}
\newcommand{\VC}{\hy{\V}{\Cat}}
\newcommand{\VOC}{\iOC{\V}}
\newcommand{\VOmC}{\iOmC{\V}}
\newcommand{\VOpC}{\iOpC{\V}}
%
%
\newcommand{\We}{\mathfrak{W}}
\newcommand{\fG}{\mathfrak{G}}
%
%
\newcommand{\bDel}{\mathbf{\Delta}}
\newcommand{\bF}{\mathbb F}
\newcommand{\Fp}{\bF_{p}}
\newcommand{\bL}{\mathbb L}
\newcommand{\tL}{\tilde{\bL}}
\newcommand{\bN}{\mathbb N}
\newcommand{\bQ}{\mathbb Q}
\newcommand{\bZ}{\mathbb Z}
\newcommand{\bz}{\mathbf 0}
\newcommand{\bo}{\mathbf 1}
\newcommand{\bt}{\mathbf 2}
\newcommand{\bn}{\mathbf n}
%
%
\newcommand{\bS}[1]{{\mathbf S}^{#1}}
\newcommand{\bSp}[2]{\bS{#1}_{({#2})}}
\newcommand{\be}[1]{{\mathbf e}^{#1}}
\newcommand{\gS}[1]{{\EuScript S}^{#1}}
\newcommand{\hT}{\hat{T}}
%
%
\newcommand{\pis}{\pi_{\ast}}
\newcommand{\pisp}{\pi_{\ast+1}}
\newcommand{\hpi}{\hat{\pi}_{1}}
\newcommand{\pin}{\pi^{\natural}}
\newcommand{\pinat}[1]{\pin_{#1}}
\newcommand{\pn}{\pi_{\natural}}
\newcommand{\pun}[1]{\pi_{\natural\,{#1}}}
\newcommand{\punat}[2]{\pun{#1}^{#2}}
\newcommand{\pinatt}[2]{\pi_{#1}^{\natural\,({#2})}}
\newcommand{\punatt}[2]{\pi^{#1}_{\natural\,({#2})}}
\newcommand{\pinA}[2]{\pi^{({#1})}_{#2}}
\newcommand{\piA}{\pi^{\A}}
\newcommand{\pinB}[1]{\pinA{\B}{#1}}
\newcommand{\piB}{\pinB{}}
%
%
\newcommand{\ma}{mapping algebra}
\newcommand{\Ama}{$\A$-\ma}
\newcommand{\Bma}{$\B$-\ma}
\newcommand{\VAma}{\ww{\VA}-\ma}
\newcommand{\Et}[1]{$E^{#1}$-term}
\newcommand{\Ett}{\Et{2}\ }
\newcommand{\Pa}{$\Pi$-algebra}
\newcommand{\PiA}{\Pi_{\A}}
\newcommand{\PuA}{\Pi^{\A}}
\newcommand{\PiB}{\Pi_{\B}}
\newcommand{\PihA}{\Pi_{\hA}}
\newcommand{\PiC}{\Pi_{\C}}
\newcommand{\PiD}{\Pi_{\D}}
\newcommand{\PAa}{$\PiA$-algebra}
\newcommand{\PuAa}{$\PuA$-algebra}
\newcommand{\sPAa}{secondary \ww{\PiA}-algebra}
\newcommand{\PBa}{$\PiB$-algebra}
\newcommand{\PAlg}{\Alg{\Pi}}
\newcommand{\PAAlg}{\Alg{\PiA}}
\newcommand{\PuAAlg}{\Alg{\PuA}}
\newcommand{\PBAlg}{\Alg{\PiB}}
\newcommand{\TAlg}{\Alg{\Theta}}
%
%
\newcommand{\BL}{B\Lambda}
\newcommand{\EK}[4]{E\sp{{#1}}\sb{#2}({#3},{#4})}
\newcommand{\EL}[2]{\EK{\Lambda}{}{#1}{#2}}
\newcommand{\ECL}[2]{\EK{\Lambda}{\C}{#1}{#2}}
\newcommand{\EDL}[2]{\EK{\Lambda}{\D}{#1}{#2}}
\newcommand{\EC}[2]{\EK{}{\C}{#1}{#2}}
\newcommand{\ts}{\tilde{s}}
\newcommand{\tDe}{\tilde{\Delta}}
%
%
\newcommand{\Au}{A^{\bullet}}
\newcommand{\As}{A_{\ast}}
\newcommand{\Bd}{B_{\bullet}}
\newcommand{\Cu}{C^{\bullet}}
\newcommand{\Ed}{\E_{\bullet}}
\newcommand{\Fu}{F^{\bullet}}
\newcommand{\Gd}{G_{\bullet}}
\newcommand{\Md}{\M_{\bullet}}
\newcommand{\Nd}{\N_{\bullet}}
\newcommand{\Pd}{\PP_{\bullet}}
\newcommand{\Pu}{\PP\sp{\bullet}}
\newcommand{\Qd}{\QQ_{\bullet}}
\newcommand{\Ru}{R\sp{\bullet}}
\newcommand{\Vd}{V_{\bullet}}
\newcommand{\Wd}{W_{\bullet}}
\newcommand{\Wu}{W^{\bullet}}
\newcommand{\Xd}{X_{\bullet}}
\newcommand{\Xu}{X^{\bullet}}
\newcommand{\Yd}{Y_{\bullet}}
\newcommand{\Yid}[1]{Y_{\bullet}\q{#1}}
\newcommand{\Zd}{Z_{\bullet}}
%
%
\newcommand{\Tss}{T_{\ast\ast}}
%
%
\newcommand{\bd}{\mathbf{d}_{0}}
\newcommand{\co}[1]{c({#1})_{\bullet}}
\newcommand{\cu}[1]{c({#1})^{\bullet}}
\newcommand{\fk}[1]{f\sb{#1}}
\newcommand{\q}[1]{^{({#1})}}
\newcommand{\qk}[1]{q_{#1}}
\newcommand{\li}[1]{_{({#1})}}
%
%
\newcommand{\fff}{\mathfrak{f}}
\newcommand{\fg}{\mathfrak{g}}
\newcommand{\fM}{\mathfrak{M}}
\newcommand{\fMA}{\fM_{\A}}
\newcommand{\fuMA}{\fM^{\A}}
\newcommand{\fMB}{\fM_{\B}}
\newcommand{\fPi}{\mathbf{\Pi}}
\newcommand{\fPA}[2]{\fPi_{\A}({#1})_{#2}}
\newcommand{\fPAd}[1]{\fPA{#1}{\bullet}}
\newcommand{\fV}{\mathfrak{V}}
\newcommand{\fVd}{\fV_{\bullet}}
\newcommand{\fVu}{\fV^{\bullet}}
\newcommand{\fW}{\mathfrak{W}}
\newcommand{\fWd}{\fW_{\bullet}}
\newcommand{\fX}{\mathfrak{X}}
\newcommand{\fY}{\mathfrak{Y}}
\newcommand{\fZ}{\mathfrak{Z}}
\setcounter{section}{-1}
%
%
\title{Stems and spectral sequences}
\date{April 1, 2010; revised August 9, 2010}
\author{Hans-Joachim Baues}
\address{Max-Planck-Institut f\"{u}r Mathematik\\ 
Vivatsgasse 7\\ 53111 Bonn, Germany}
\email{baues@mpim-bonn.mpg.de}
\author{David Blanc}
\address{Department of Mathematics\\ University of Haifa\\ 31905 Haifa, Israel}
\email{blanc@math.haifa.ac.il}
\begin{abstract}
We introduce the category \w{\Pstem[n]} of $n$-stems, with a functor
\w{\PP[n]} from spaces to \w[.]{\Pstem[n]} 
This can be thought of as the $n$-th order homotopy groups of a space.
We show how to associate to each simplicial $n$-stem
\w{\Qd} an \ww{(n+1)}-truncated spectral sequence. Moreover, if
\w{\Qd=\PP[n]\Xd} is the Postnikov $n$-stem of
a simplicial space \w[,]{\Xd} the truncated spectral sequence for
\w{\Qd} is the truncation of the usual homotopy spectral
sequence of \w[.]{\Xd} Similar results are also proven for cosimplicial
$n$-stems. They are helpful for computations, since $n$-stems
in low degrees have good algebraic models.
\end{abstract}
\maketitle

%
%
\sect{Introduction}

Many of the spectral sequences of algebraic topology arise as the
homotopy spectral sequence of a (co)simplicial space \wh including
the spectral sequence of a double complex, the (stable or
unstable) Adams spectral sequence, the Eilenberg-Moore spectral
sequence, and so on (see \S \ref{egert}). 
Given a simplicial space \w[,]{\Xd} the \Ett of its homotopy spectral
sequence has the form \w[,]{E^{2}_{s,t}=\pi_{s}\pi_{t}\Xd} so it may
be computed by applying the homotopy group functor dimensionwise to 
\w[.]{\Xd} 

In this paper we show that the higher terms of this spectral sequence are 
obtained analogously by applying 'higher homotopy group' functors to
\w[.]{\Xd} These functors are given explicitly in the form of certain
\emph{Postnikov stems}, defined in Section \ref{cpstem}; the  
Postnikov $0$-stem of a space is equivalent to its homotopy groups.

We then show how the \Et{r}\ of the homotopy spectral sequence of a
simplicial space \w{\Xd} can be described in terms of the
\ww{(r-2)}-Postnikov stem of \w[,]{\Xd} for each \w{r\geq 2} (see
Theorem \ref{tert}) \wh and similarly for the homotopy spectral
sequence of a cosimplicial space \w{\Xu} (see Theorem \ref{tcert}). 

As an application for the present paper, in \cite{BBlaH} we generalize 
the first author's result with Mamuka Jibladze (in \cite{BJiblSD}),
which shows that the \Et{3}\ of the stable Adams spectral sequence can
be identified as a certain secondary derived functor\w[.]{\Ext}
We do this by showing how to define in general the \emph{higher order
  derived functors} of a continuous functor \w[,]{F:\C\to\Ta} by
applying $F$ to a simplicial resolution \w{\Wd} in $\C$, and taking
Postnikov $n$-stems of \w[.]{F\Wd}

\begin{mysubsection}[\label{snac}]{Notation and conventions}
The category of pointed connected topological spaces will be denoted
by \w[;]{\Ta} that of pointed sets by \w[;]{\Seta} that of groups by
\w[.]{\Gp} For any category $\C$, \w{s\C} denotes the category of
simplicial objects over $\C$, and \w{c\C} that of cosimplicial objects
over $\C$. However, we abbreviate \w{s\Set} to \w[,]{\Ss} \w{s\Seta}
to \w[,]{\Sa} and \w{s\Gp} to $\G$. The constant (co)simplicial object
on an object \w{X\in\C} is written \w{\co{X}\in s\C} 
(respectively, \w[).]{\cu{X}\in c\C} For any small indexing category
$I$, the category of functors \w{I\to\C} is denoted by \w[.]{\C^{I}}
\end{mysubsection}

\begin{ack}
We wishe to thank the referee for his or her careful reading of the
paper and helpful comments on it.
\end{ack}

%
%
\sect{Postnikov stems}
\label{cpstem}

The Postnikov system of a topological space (or simplicial set) $X$ is
the tower of fibrations:
\begin{myeq}[\label{eqpostow}]
\dotsc~\to~\Po{n+1}X~\xra{p^{n+1}}~\Po{n}X~\xra{p^{n}}~
\Po{n-1}X~\dotsc\Po{1}X~\xra{p^{1}}~\Po{0}X~,
\end{myeq}
\noindent equipped with maps \w{q^{n}:X\to\Po{n}X} (with
\w[),]{p^{n}\circ q^{n} =q^{n-1}} which induce isomorphisms on   
homotopy groups in degrees \w[.]{\leq n} Here \w{\Po{n}X} is $n$-coconnected
(that is, \w{\pi_{i}\Po{n}X=0} for \w[)]{i>n} and \w{\pi_{i}p^{n}} is
an isomorphism for \w[.]{i<n} The fiber of the map
\w{p^{n}:\Po{n}X\to\Po{n-1}X} is the Eilenberg-Mac~Lane space  
\w[,]{K(\pi_{n}X,n)} so the fibers are determined up to homotopy by
\w[.]{\pis X} Thus a generalization of the homotopy groups of $X$ is
provided by the following notion:

\begin{defn}\label{dpstem}
For any \w[,]{n\geq 0} a \emph{Postnikov $n$-stem} in \w{\Ta} is a tower:
\begin{myeq}[\label{eqponstem}]
\QQ~:=~\left(\dotsc~\to~\Qk{k+1}~\xra{\qk{k+1}}~\Qk{k}~\xra{\qk{k}}~
\Qk{k-1}~\dotsc~\Qk{0}\right)
\end{myeq}
\noindent in \w[,]{\Ta^{(\bN,\leq)}} in which \w{\Qk{k}} is 
\ww{(k-1)}-connected and \ww{(n+k)}-coconnected (so that
\w{\pi_{i}(\Qk{k})=0} for \w{i<k} or \w[)]{i>n+k} and
\w{\pi_{i}(\qk{k})} is an isomorphism for \w[.]{k\leq i<n+k}
Here \w{(\bN,\leq)} is the usual linearly ordered category of the natural
numbers. The space \w{\Qk{k}} is called the $k$-th \emph{$n$-window}
of $\QQ$. 

Such an $n$-stem is thus a collection of overlapping \ww{(k-1)}-connected
\ww{n+k}-types, which may be depicted for \w{n=2} as follows:
%
$$
\begin{array}{rrrrrrrrr}
\dotsc & \ast & \ast & \ast & & & & \\
& & \ast & \ast & \ast & & & \\
& & & \ast & \ast & \ast & & \\
& & & & \ast & \ast & \ast & \dotsc
\end{array}
$$
%
\noindent where each row exhibits the \w{n+1} non-trivial homotopy groups
(denoted by $\ast$) of one $n$-window, and all those in the $i$-th column
(corresponding to \w[)]{\pi_{i}} are isomorphic.

We denote by \w{\Pstem[n]} the full subcategory of Postnikov $n$-stems
in the functor category \w{\Ta^{(\bN,\leq)}} (with model
category structure on the latter as in \cite[11.6]{PHirM}). Thus the
morphisms in \w{\Pstem[n]} are given by strictly commuting maps of
towers, and \w{f:\QQ\to\QQ'} is a weak equivalence (respectively, a
fibration) if each \w{\fk{k}:\Qk{k}\to\Qk{k}'} is such.
This lets us define the homotopy category of Postnikov $n$-stems, 
\w[,]{\ho\Pstem[n]} as a full sub-category of \w[.]{\ho\Ta^{(\bN,\leq)}}

The category \w{\Pstem[n]} is pointed, has products, and is
equipped with canonical functors 

\mydiagram[\label{eqpstems}]{
\Ta \ar[rd]^<<<<<<<<<<<<<<<<<<{\PP[n]} 
\ar@/^0.5pc/[rrd]^<<<<<<<<<<<<<<<<<<<<{\PP[n-1]}
\ar@/^1pc/[rrrrd]^<<<<<<<<<<<<<<<<<<<<<<<<<<<<<<<<<<<<{\PP[0]} &&&& \\
 & \dotsc~~  \Pstem[n]\ar[r]_<<<<{\bP{n-1}}~~ & 
~\Pstem[n-1]~\ar[r]_<<<<<{\bP{n-2}}~  & ~\dotsc~\ar[r]_<<<<{\bP{0}} &~\Pstem[0]
}
\noindent which preserve products and weak equivalences.
\end{defn}

\begin{remark}\label{rpstems}
The sequence of functors \wref{eqpstems} is described by a commuting
diagram, in which we may take all maps to be fibrations: 
\mydiagram[\label{eqpstem}]{
& & & & & \\
\dotsc \ar[r] & \Qnk{n+k+1}{k+1} \ar[ru] \ar@{.}[u] \ar[d]_{p^{n}_{k+1}} 
\ar[r]^{q^{n}_{k+1}} &
\Qnk{n+k}{k} \ar[ru] \ar@{.}[u] \ar[d]_{p^{n}_{k}} \ar[r]^{q^{n}_{k}} & 
\Qnk{n+k-1}{k-1} \ar[ru] \ar@{.}[u] \ar[d]^{p^{n}_{k-1}} \ar[r] \ar[ru] & 
\dotsc \ar[r] &  \Qnk{n}{0} \ar[d]^{p^{n}_{0}} \\
\dotsc \ar[r] & \Qnk{n+k}{k+1} \ar[d]_{p^{n-1}_{k+1}} \ar[r]^{q^{n-1}_{k+1}} 
\ar[ru]^{r^{n-1}_{k+1}} &
\Qnk{n+k-1}{k} \ar[d]_{p^{n-1}_{k}} \ar[r]^{q^{n-1}_{k}}\ar[ru]^{r^{n-1}_{k}} & 
\Qnk{n+k-2}{k-1} \ar[d]^{p^{n-1}_{k-1}} \ar[r] \ar[ru] & 
\dotsc \ar[r] & \Qnk{n-1}{0} \ar[d]^{p^{n}_{0}} \\
\dotsc \ar[r] & \Qnk{n+k-1}{k+1} \ar@{.}[d] \ar[r]^{q^{n-2}_{k+1}} 
\ar[ru]^{r^{n-2}_{k+1}} &
\Qnk{n+k-2}{k} \ar@{.}[d] \ar[r]^{q^{n-2}_{k}} \ar[ru]^{r^{n-2}_{k}} & 
\Qnk{n+k-3}{k-1} \ar@{.}[d] \ar[r] \ar[ru] & 
\dotsc \ar[r] & \Qnk{n-2}{0} \ar@{.}[d] \\
\dotsc \ar[r] & \Qnk{k+1}{k+1} \ar[r]^{q^{0}_{k+1}} &
\Qnk{k}{k} \ar[r]^{q^{0}_{k}} & \Qnk{k-1}{k-1} \ar[r] & 
\dotsc \ar[r] & \Qnk{0}{0}  
}
\noindent Here \w{\pi_{i}\Qnk{n}{k}=0} for \w{i<k} or \w[,]{i>n}
and all maps induce isomorphisms in \w{\pi_{i}} whenever possible. Thus:

\begin{enumerate}
\renewcommand{\labelenumi}{(\alph{enumi})}
\item The $k$-th column (from the right) is the Postnikov tower for 
\w[.]{\Qnk{}{k}:=\lim_{n}\Qnk{n}{k}}
\item The diagonals are the dual Postnikov system of
  connected covers for \w[.]{\Qnk{j}{0}}
\item The $n$-th row (from the bottom) is a Postnikov $n$-stem.
\item In particular, each space in the $0$-stem (the bottom row) is an 
Eilenberg-Mac~Lane space, and the maps \w{q^{0}_{k}} are nullhomotopic. Thus 
the homotopy type of the bottom line in \w{\ho\Pstem[0]} is  
determined by the collection of homotopy groups 
\w[.]{\{\pi_{k}\Qnk{k}{k}\}_{k=0}^{\infty}}
\end{enumerate}
\end{remark}

\begin{defn}\label{drpstem}
The motivating example of a Postnikov $n$-stem is a \emph{realizable}
one, associated to a space \w[,]{X\in\Ta} and denoted by
\w[,]{\PP[n]X} with \w[.]{\PPk{k}{n}{X}:=\Po{n+k}\wk{X}{k}} 
As usual, \w{\wk{Y}{k}} denotes the \ww{(k-1)}-connected cover of
a space \w[.]{Y\in\Ta} Each fibration \w{\qk{k}:\PPk{k}{n}{X}\to\PPk{k-1}{n}{X}} 
fits into a commuting triangle of fibrations:
\mydiagram[\label{eqpstm}]{
\Pnk{n+k+1}{X}{k+1} \ar[rd]^{p} \ar[rr]^{q} && \Pnk{n+k}{X}{k} \\
& \Pnk{n+k}{X}{k+1} \ar[ru]^{r} &
}
\noindent in which the maps $p$ and $r$ are the fibration of
\wref{eqpostow} and the covering map, respectively. See \cite[\S 10.5]{BBlaC} 
for a natural context in which non-realizable Postnikov $n$-stems arise. 
\end{defn}

\begin{mysubsection}[\label{egpoststem}]{Examples of stems}
The functor \w{\PP[0]_{\ast}:\Ta\to\ho\Pstem[0]} induced by
\w{\PP[0]} is equivalent to the homotopy group functor:
in fact, the homotopy groups of a space define a functor 
\w{\pis:\Ta\to\K} into the product category 
\w[,]{\K:=\prod_{i=0}^{\infty}\,\K_{i}} where \w[,]{\K_{0}=\Seta}
\w[,]{\K_{1}=\Gp} and \w[,]{\K_{i}=\Abgp} for \w[.]{i\geq 2} Moreover,
there is an equivalence of categories 
\w[,]{\vartheta:\K\equiv\ho\Pstem[0]} such that the functor
\w{\PP[0]_{\ast}} is equivalent to the composite functor
\w[.]{\vartheta\circ\pis:\Ta\to\K}  

Similarly, the functor \w{\Ta\to\ho\Pstem[1]} induced by
\w{\PP[1]} is equivalent to the secondary homotopy group functor of
\cite[\S 4]{BMuroS}, in the sense that each secondary homotopy group
\w{\pi_{n,\ast}X} completely determines the $n$-th $1$-window of $X$.
Using the results on secondary homotopy groups in \cite{BMuroS}, one
obtains a homotopy category of algebraic $1$-stems which is equivalent
to \w[.]{\ho\Pstem[1]} 

A category of algebraic models for $2$-stems is only partially known.
The homotopy classification of \ww{(k-1)}-conected \ww{(k+2)}-types is
described for all $k$ in \cite{BauHH}; this theory can be used to
classify homotopy types of Postnikov $2$-stems. 
\end{mysubsection}

%
%
\sect{The spectral sequence of a simplicial space}
\label{cssss}

We begin with the construction of the homotopy spectral sequence for a
simplicial space (cf.\ \cite{QuiS}, \cite[Theorem B.5]{BFrieH},
and \cite[X,\S 6]{BKanH}), using the version given by Dwyer,
Kan, and Stover in \cite[\S 8]{DKStB} (see also \cite[\S 2,5]{BousCR},
\cite{BousH}, and \cite[\S 3.6]{DKStE}). For this purpose, we require
some explicit constructions for the $E^{2}$-model category of
simplicial spaces.

\begin{defn}\label{dnc}
Given a simplicial object \w[,]{\Xd\in s\C} over a complete pointed
category $\C$, for each \w{n\geq 1} define its $n$-\emph{cycles}
object to be 
$$
Z_{n}\Xd~:=~\{ x\in X_{n}\,|\ d_{i}x=\ast \ \text{for}\
i=0,\dotsc,n\}~.
$$
\noindent Similarly, the the $n$-\emph{chains} object for \w{\Xd} is 
$$
C_{n}\Xd~:=~\{ x\in X_{n}\,|\ d_{i}x=\ast \ \text{for}\ i=1,\dotsc,n\}
$$
\noindent Set \w[.]{Z_{0}\Xd:=X_{0}} We denote the map
\w{d_{0}\rest{C_{n}\Xd}:C_{n}\Xd\to Z_{n-1}\Xd} by \w[.]{\bd^{X_{n}}}
\end{defn}

\begin{notn}\label{nloop}
For any non-negatively graded object \w[,]{T_{\ast}} we write
\w{\Omega T_{\ast}} for the graded object with 
\w{(\Omega T_{\ast})_{j}:=T_{j+1}} for all \w[.]{j\geq 0} The notation
is motivated by the natural isomorphism of graded groups
\w{\pis\Omega X\cong\Omega(\pis X)} for \w[.]{X\in\Ta} 
\end{notn}

\begin{defn}\label{dnhg}
Now assume that $\C$ is a pointed model category of spaces, such as
\w{\Ta} or $\G$, and \w{\Xd} is a Reedy fibrant simplicial object
over $\C$ \wh that is, for each \w[,]{n\geq 1} the universal face map 
\w{\delta_{n}:X_{n}\to M_{n}\Xd} into the $n$-th matching object of
\w{\Xd} is a fibration (see \cite[15.3]{PHirM}). The map
\w{\bd=\bd^{X_{n}}} then fits into a fibration sequence in $\C$:   
\begin{myeq}[\label{eqfibseq}]
\dotsb \Omega Z_{n}\Xd\to Z_{n+1}\Xd\xra{j^{\Xd}_{n+1}}C_{n+1}\Xd
\xra{\bd^{X_{n+1}}} Z_{n}\Xd
\end{myeq}
\noindent (see \cite[Prop.\ 5.7]{DKStB}). 

For each \w[,]{n\geq 0} the $n$-th \emph{natural homotopy group} of
the simplicial space \w[,]{\Xd} denoted by
\w[,]{\pin_{n}\Xd=\pin_{n,\ast}\Xd} the cokernel of the map 
\w{(\bd^{X_{n+1}})_{\#}} (induced on homotopy groups by \w[).]{\bd^{X_{n+1}}}
Note that the cokernel of a maps of groups or pointed sets is
generally just a pointed set.

We thus have an exact sequence of graded groups:  
\begin{myeq}\label{eqnathom}
\pis C_{n+1}\Xd~\xra{(\bd^{X_{n+1}})_{\#}}~\pis Z_{n}\Xd~\xra{\hvart_{n}}~ 
\pin_{n,\ast}{\Xd}\to 0~.
\end{myeq}
\noindent Together the groups \w{(\pin_{n,k}\Xd)_{n,k=0}^{\infty}}
constitute the \emph{bigraded homotopy groups} of \cite[\S 5.1]{DKStB}.  
\end{defn}

\begin{mysubsection}[\label{scsles}]{Construction of the spiral sequence}
Applying the functor \w{\pis} to the fibration sequence
\wref{eqfibseq} yields a long exact sequence, with connecting homomorphism
\w[.]{\partial_{\#}:\Omega\pis Z_{n}\Xd=\pis\Omega Z_{n}\Xd\to\pis Z_{n+1}\Xd}
Note that the inclusion \w{\iota:C_{n}\Xd\hra X_{n}}
induces an isomorphism \w{\iota_{\star}:\pis C_{n}\Xd\cong C_{n}(\pis\Xd)}
for each \w{n\geq 0} (see \cite[Prop.\ 2.7]{BlaCW}). From
\wref{eqnathom} we see that: 
$$
\Omega\pin_{n}\Xd=\Omega\Cok(\bd^{X_{n+1}})_{\#}\cong\Image\partial_{\#}
\cong \Ker(j^{\Xd}_{n+1})_{\#}\subseteq\pis Z_{n+1}\Xd~,
$$
\noindent so we obtain a commutative diagram with exact rows and columns: 
\mydiagram[\label{eqrowcol}]{
& 0 \ar[d] & 0 \ar[d] & 0 \ar[d] & & \\ 
0 \ar[r] & \Ker(j_{n})_{\ast}~ \ar@{^{(}->}[r] \ar@{^{(}->}[d] & 
B_{n+1}\Xd \ar@{^{(}->}[d]
\ar@{->>}[r]^>>>>>{(j_{n})_{\ast}} & B_{n+1}\pis\Xd\ar[r] \ar@{^{(}->}[d] & 
0\ar[d] & \\
0 \ar[r] & \Omega\pin_{n-1}\Xd~ \ar@{^{(}->}[r]^{\ell_{n-1}} 
\ar@{->>}[d] \ar[rd]^{s_{n}} & 
\pis Z_{n}\Xd \ar[d]^{\hvart_{n}} \ar[r]^{(j_{n}^{\Xd})_{\#}} & 
Z_{n}\pis\Xd \ar@{->>}[r] \ar[d]^{\vart_{n}} & \Cok h_{n} \ar[r] \ar[d]^{=} & 0 \\
0 \ar[r] & \Ker h_{n}~ \ar@{^{(}->}[r] \dto & \pin_{n}\Xd \ar[d] \ar[r]^{h_{n}} & 
\pi_{n}\pis\Xd \ar@{->>}[r] \ar[d] & \Cok h_{n}\ar[r] \ar[d] & 0 \\
 & 0 & 0 & 0 & 0 &
}
\noindent in which 
\w{B_{n+1}\Xd:=\Image(\bd^{X_{n+2}})_{\#}\subseteq\pis Z_{n}\Xd} and 
\w{B_{n+1}\pis X_{n+2}:=\Image\bd^{\pis X_{n+2}}} are the respective
boundary objects. Note that the map 
\w{(j_{n}^{\Xd})_{\#}:\pis Z_{n}\Xd\to\pis C_{n}\Xd} induced by the
inclusion \w{j_{n}^{\Xd}} of \wref{eqfibseq} above in fact factors
through \w[,]{Z_{n}\pis\Xd} as indicated in the middle row of
\wref[.]{eqrowcol} 

This defines the map of graded groups
\w[.]{h_{n}:\pin_{n}\Xd\to\pi_{n}(\pis\Xd)} Note that for \w{n=0} the map
\w{\hat{\iota}_{\star}} is an isomorphism, so \w{h_{0}} is, too.
The map \w{s_{n}:\Omega\pin_{n-1}\Xd\to\pin_{n}\Xd} is the composite
of the inclusion \w{\ell_{n-1}:\Ker(j^{\Xd}_{n})_{\#}\hra\pis Z_{n}\Xd} with 
the quotient map \w{\hvart_{n}:\pis Z_{n}\Xd\to\pin_{n}\Xd} 
of \wref[,]{eqnathom} using the natural identification of
\w{\Omega\pin_{n}\Xd} with \w[.]{\Ker(j^{\Xd}_{n+1})_{\#}}

The map \w{\partial_{n+2}:\pi_{n+2}\pis\Xd\to\Omega\pin_{n}\Xd}
is induced by the composite 
\begin{myeq}[\label{eqpartial}]
Z_{n+2}\pis\Xd\subseteq C_{n+2}\pis\Xd\cong\pis C_{n+2}\Xd
\xra{(\bd^{X_{n+2}})_{\#}}\pis Z_{n+1}\Xd~,
\end{myeq}
\noindent which actually lands in \w{\Ker(j^{\Xd}_{n+1})_{\#}} by the
exactness of the long exact sequence for the fibration \wref[.]{eqfibseq}

These maps \w[,]{s_{n}} \w[,]{h_{n}} and \w{\partial_{n}} fit into 
a \emph{spiral long exact sequence:} 
\begin{myeq}[\label{eqspiral}]
\begin{split}
\ldots~\to~\Omega\pinat{n-1}\Xd & ~\xra{s_{n}}~\pinat{n}\Xd~\xra{h_{n}}~
\pi_{n}\pis\Xd~\xra{\partial_{n}}~\Omega\pinat{n-2}\Xd\\ 
&~\xra{s_{n-1}}~\pinat{n-1}\Xd~\to~\ldots~\to~\pinat{0}\Xd~\xra{\cong}~
\pi_{0}\pis\Xd 
\end{split}
\end{myeq}
\noindent (cf.\ \cite[8.1]{DKStB}).
\end{mysubsection}

\begin{mysubsection}[\label{sssss}]{The spectral sequence of a simplicial space}
For any simplicial space \w{\Xd\in s\Ta} (or bisimplicial set), Bousfield
and Friedlander showed that there is a first-quadrant spectral sequence 
of the form
\begin{myeq}[\label{eqbfried}]
E^{2}_{s,t}=\pi_{s}\pi_{t}\Xd~\Ra~\pi_{s+t}\|\Xd\|~,
\end{myeq}
\noindent where \w{\|\Xd\|\in\Ta} is the realization (or the diagonal,
in the case of \w[).]{\Xd\in s\Sa} The spectral sequence is always
defined, but \w{\Xd} must satisfy certain ``Kan conditions'' to
guaranteee \emph{convergence} \wh see \cite[Theorem B.5]{BFrieH}. 

In \cite[\S 8.4]{DKStB}, Dwyer, Kan and Stover showed that \wref{eqbfried}
coincides up to sign, from the \Ett on, with the spectral
sequence associated to the exact couple of \wref[,]{eqfibseq} which we
call the \emph{spiral spectral sequence} for \w[.]{\Xd}

If we assume that each \w{X_{n}} is connected, by taking loops (or
applying Kan's functor $G$, if \w[),]{\Xd\in s\Sa} we may replace
\w{\Xd} by a bisimplicial group \w[,]{G\Xd\in s\G} and then
\wref{eqbfried} becomes the spectral sequence of \cite{QuiS}.
\end{mysubsection}

%
%
\sect{Simplicial stems and truncated spectral sequences}
\label{csstss}

As noted in \S \ref{egpoststem}, the \Ett of any of the above equivalent
spectral sequences for a simplicial space \w{\Xd} is determined
explicitly by the simplicial $0$-stem of \w[.]{\Xd}

Our goal is to extend this description to the higher terms of the
spectral sequence. For this purpose, fix \w[,]{n\geq 0} and consider
a simplicial Postnikov $n$-stem \w{\Qd} (which need not be
realizable as \w{\PP[n]\Xd} for some simplicial space \w[).]{\Xd} 
This is equivalent to having a collection of simplicial spaces
\w{\Qdnk{n+k}{k}} for each \w[,]{k\geq 0} equipped with maps as in
\wref[,]{eqponstem} with \w{\pi_{i}\Qdnk{n+k}{k}=0} for \w{i<k} or \w[.]{i>n+k}

We assume that \w{\Qd} is \emph{Reedy fibrant} in the sense that for
each \w[,]{k\geq 0} the simplicial space \w{\Qdnk{n+k}{k}} is Reedy fibrant.
In this case, the ``$n$-stem version'' of the spiral long exact
sequence is defined as follows: for each \w[,]{t,i,k\geq 0} set 
\w{\pinatt{t,i}{k,n}\Qd:=\pinat{t,i+k}\Qdnk{n+k}{k}} and  
\begin{myeq}[\label{eqtruncgps}]
\pinA{k,n}{i}\Qd~:=~\pi_{i+k}\Qdnk{n+k}{k}~=~\begin{cases} 
\pi_{i+k}\Qd &\text{if}~~0\leq i\leq n\\
0 & \text{otherwise.}
\end{cases} 
\end{myeq}
\noindent Note that the \ww{(i+k)}-th homotopy group \w{\pi_{i+k}\Qd} 
of a Postnikov $n$-stem \w{\Qd} is well-defined, and coincides with
\w{\pi_{i+k}\Xd} for \w{0\leq i\leq n} when \w[.]{\Qd=\PP[n]\Xd}

\begin{defn}\label{dsns}
The collection of long exact sequences \wref{eqspiral} for 
\w{\Qdnk{n+k}{k}} (indexed by \w[):]{k\geq 0} 
\mydiagram[\label{eqpspiral}]{
\dotsc \Omega\pinatt{t-1,\ast}{k,n}\Qd~\ar[r]^<<<<<{s_{t}\q{k,n}} &
\pinatt{t,\ast}{k,n}\Qd~\ar[r]^<<<<<{h_{t}\q{k,n}} &
\pi_{t}\pinA{k,n}{\ast}\Qd~\ar[r]^<<<<<{\partial_{t}\q{k,n}} &
\Omega\pinatt{t-2,\ast}{k,n}\Qd\dotsc~,
}
\noindent together with the maps between adjacent $k$-windows induced
by the map $q$ in \wref[,]{eqpstem} will be called the 
\emph{spiral $n$-system} of \w[.]{\Qd} When \w[,]{\Qd=\PP[n]\Xd} we
will refer to this simply as the spiral $n$-system of \w[.]{\Xd}
\end{defn}

\begin{remark}\label{rsns}
Using the exactness of \wref[,]{eqpspiral} definition
\wref{eqtruncgps} implies that: 
\begin{myeq}[\label{eqvannatgp}]
\pinatt{t,i}{k,n}\Qd=\pinat{t,i}\Qdnk{n+k}{k}~=~0\hs\text{for}~~~i>n~,
\end{myeq}
\noindent by induction on \w[.]{t\geq 0} Note, however, that while the
groups \w{\pinA{k,n}{i}\Qd} are explicitly described by
\wref[,]{eqtruncgps} the dependence of \w{\pinatt{t,i}{k,n}\Qd} on $k$
and $n$ requires more care. 
\end{remark}

\begin{mysubsection}[\label{setht}]{The \Et{2}\ of the spectral sequence}
The spiral $0$-system of a simplicial Postnikov $0$-stem \w{\Qd}
reduces to a series of isomorphisms 
\w{h_{t}:\pinatt{t,\ast}{k,0}\Qd\cong\pi_{t}\pinA{k,0}{\ast}\Qd} 
(for each \w[).]{k\geq 0} When \w{\Qd=\PP[0]\Xd} is the Postnikov
$0$-stem of a simplicial space \w[,]{\Xd} this allows us to identify
the \ww{E^{2}_{t,k}}-term of the spiral spectral sequence for
\w[,]{\Xd} which is:
$$
\pi_{t}\pi_{k}\Xd~=~\pi_{t}\pi_{k}\Po{0+k}\wk{\Xd}{k}~=~
\pi_{t}\pi_{k}\PPk{k}{0}{\Xd}~=~\pi_{t}\pinA{k,0}{\ast}\PP[0]\Xd~=~
\pi_{t}\pinA{k,0}{\ast}\Qd,
$$
\noindent with \w[.]{\pinatt{t,\ast}{k,0}\Qd=\pinatt{t,\ast}{k,0}\PP[0]\Xd} 
\end{mysubsection}

The first interesting case is the spiral $1$-system, for which we have:

%
%
\begin{prop}\label{pett}
The \Et{3}\ of the spiral spectral sequence for a simplicial space
\w{\Xd} is determined by the spiral $1$-system of \w[.]{\Xd}  
In fact, \w{d^{2}_{t,k}} may be identified with  
\w[,]{\partial_{t}\q{k,1}:\pi_{t}\pi_{k}\Xd\to\Omega\pinatt{t-2,0}{k,1}\Xd}
while \w{E^{3}_{t,k}} is the image of the composite map
\mydiagram[\label{eqetht}]{
\pinatt{t,0}{k,1}\Xd \ar[r]^<<<<{h_{t}\q{k,1}}&
\pi_{t}\pi_{k}\Xd~\cong~\pi_{t}\pi_{1}\q{k-1,1}\Xd~& 
\ar[l]^<<<<{\cong}_<<<<{h_{t}\q{k-1,1}}
~\pinatt{t,1}{k-1,1}\Xd \ar[r]^{s_{t+1}\q{k-1,1}} &
\pinatt{t+1,0}{k-1,1}\Xd~.
}
\end{prop}

Observe that \wref{eqetht} involves maps from different windows of the
spiral $1$-system, implicitly identified using the isomorphisms
induced by the map $q$ in \wref[.]{eqpstem} 

\begin{proof}
Because \w{n=1} throughout, we abbreviate
\w{\pinatt{t,i}{k,1}\Qd} to \w[,]{\pinatt{t,i}{k}\Qd} and
\w{\pinA{k,1}{i}\Qd} to \w[,]{\pinA{k}{i}\Qd} observing that 
\w{\pinA{k}{i}\Qd} is simply \w{\pi_{i+k}\Xd} for
\w[,]{i=0,1} and zero otherwise, since \w[.]{\Qd=\PP[1]\Xd} 
Thus the spiral $1$-system \wref{eqpspiral} is non-trivial for each
\w{t\geq 1} in (internal) degrees \w{i=0,1} only, and we can write it
in two rows:   
$$
\begin{array}{rrl}
0\hsm \lora\hsm  \pinatt{t,1}{k}\Qd \xra{~\cong~} 
\pi_{t}\pinA{k}{1}\Qd\hsm\lora & 0\hs\lora\hs & 
\pinatt{t-1,1}{k}\Qd \xra{~\cong~} \pi_{t-1}\pinA{k}{1}\Qd \\
\Omega\pinatt{t-1,0}{k}\Qd \xra{s_{t}} \pinatt{t,0}{k}\Qd \xra{h_{t}} 
\pi_{t}\pinA{k}{0}\Qd\hsm \xra{\partial_{t}} & 
\Omega\pinatt{t-2,0}{k}\Qd \xra{s_{t-1}}  & 
\pinatt{t-1,0}{k}\Qd \xra{h_{t-1}} \pi_{t-1}\pinA{k}{0}\Qd 
\end{array}
$$

Since \w{\Qd:=\PP[1]\Xd} is the simplicial Postnikov $1$-stem of
\w[,]{\Xd} we actually have a collection of two-row  
long exact sequences, one for each $k$-window of \w[.]{\PP[1]\Xd}

For each such $k$-window \w[,]{\PP_{k}[1]\Xd} we can use the top row
to identify 
$$
\Omega\pinatt{t,0}{k}\Qd~=~\Omega\pinat{t,0}\PP_{k}[1]\Xd~=~
\pinat{t,1}\PP_{k}[1]\Xd~=~\pinatt{t,1}{k}\Qd
$$
\noindent with 
\w[,]{\pi_{t}\pinA{k}{1}\Qd=\pi_{t}\pinA{1}{t}\PP_{k}[1]\Xd=\pi_{t}\pi_{k+1}\Xd} 
so the bottom row reduces to:

\begin{mypict}[\label{eqbottom}]{(320,40)(-30,0)
%
%
\put(0,40){$\pi_{t-1}\pi_{k+1}\Xd$}
\put(65,45){\vector(1,0){50}}
\put(80,49){$s_{t}\q{k,1}$}
\put(35,34){\vector(1,-1){21}}
\put(35,34){\vector(1,-1){23}}
\put(120,40){$\pinatt{t,0}{k}\Qd$}
\put(160,45){\vector(1,0){50}}
\put(175,49){$h_{t}\q{k,1}$}
\put(146,34){\vector(1,-1){21}}
\put(146,34){\vector(1,-1){23}}
\put(220,40){$\pi_{t}\pi_{k}\Xd$}
\put(265,45){\vector(1,0){40}}
\put(280,49){$\partial_{t}\q{k,1}$}
\put(245,34){\vector(1,-1){21}}
\put(245,34){\vector(1,-1){23}}
\put(310,40){$\pi_{t-2}\pi_{k+1}\Xd$}
%
%
\put(50,0){$\Image(s_{t}\q{k,1})$}
\put(105,15){\vector(1,1){20}}
\put(105,15){\vector(1,1){2}}
\put(170,0){$\Image(h_{t}\q{k,1})$}
\put(185,-13){$=$}
\put(170,-26){$\Ker(\partial_{t}\q{k,1})$}
\put(214,15){\vector(1,1){20}}
\put(214,15){\vector(1,1){2}}
\put(270,0){$\Image(\partial_{t}\q{k,1})$}
\put(310,15){\vector(1,1){20}}
\put(310,15){\vector(1,1){2}}
}
\end{mypict}

\quad\vsp\quad

Note that the following part of the \Et{1}\ of the exact couple
for the fibration sequence 
\w[,]{C_{n+1}\Po{1}\Omega^{i}\Xd\to Z_{n}\Po{1}\Omega^{i}\Xd}
(as in \wref[):]{eqfibseq}
\begin{equation*}
\xymatrix@R=23pt{
\pi_{1}Z_{t-1}\Po{1}\Omega^{k}\Xd \ar[r]^{(j_{t-1})_{\#}} 
\ar[ddd]^{\partial_{\ast}} &
\pi_{1}C_{t-1}\Po{1}\Omega^{k}\Xd \ar[r]^{(d_{0}^{t-1})_{\#}} & 
\pi_{1}Z_{t-2}\Po{1}\Omega^{k}\Xd \ar[r]^{(j_{t-2})_{\#}} 
\ar[rd] \ar@/_9pc/[ddd]_{\partial_{\ast}} \ar@{->>}[d] & 
\pi_{1}C_{t-2}\Xd \to \dotsc \\
&&\quad\hsn\Omega\pinatt{t-2,0}{k}\Xd=\pinatt{t-2,1}{k}\Xd \ar@{^{(}->}[dd] 
\ar[rd]^>>>>>>>>>>>{h_{t-2,1}\q{k+1,1}}_{\cong} & 
Z_{t-2}\pi_{1}\Po{1}\Omega^{k}\Xd \ar@{^{(}->}[u]_{\inc} 
\ar@{->>}[d]^{\vart_{t-2}}\\ 
&&&\pi_{t-2}\pi_{k+1}\Xd\\
\pi_{0}Z_{t}\Po{1}\Omega^{k}\Xd \ar[r]^{(j_{t})_{\#}} \ar[rd] 
\ar@{->>}[d]^{\hvart_{t}} & 
\pi_{0}C_{t}\Po{1}\Omega^{k}\Xd \ar[r]^<<<<{(d_{0}^{t})_{\#}} & 
\pi_{0}Z_{t-1}\Po{1}\Omega^{k}\Xd \ar[r]^{(j_{t-1})_{\#}} \ar[rd] 
\ar@{->>}[d]^{\hvart_{t-1}} & 
\pi_{0}C_{t-1}\Po{1}\Omega^{k}\Xd \to \dotsc\\
\pinatt{t,0}{k}\Xd \ar[rd]^{h_{t}\q{k,1}} & 
Z_{t}\pi_{k}\Xd \ar@{^{(}->}[u]_{\inc} 
\ar@{->>}[d]_{\vart_{t}} & \pinatt{t-1,0}{k}\Xd \ar[rd]^{h_{t-1}\q{k,1}} & 
Z_{t-1}\pi_{k}\Xd \ar@{^{(}->}[u]_{\inc} \ar@{->>}[d]^{\vart_{t-1}}\\ 
&\pi_{t}\pi_{k}\Xd \ar[ruuuu]^>>>>>>>>>>>>>{\partial_{t,0}\q{k,1}}&&
\pi_{t-1}\pi_{k}\Xd
}
\end{equation*}
\noindent is naturally isomorphic to the exact couple for 
\w[,]{C_{n+1}\Omega^{k}\Xd\to Z_{n}\Omega^{k}\Xd}
since \w{C_{n+1}} and \w{Z_{n}}  are limits, so they
commute with \w[,]{\Po{1}} and then
\w[,]{\pi_{1}\Po{1}Z_{t-1}\Omega^{k}\Xd\cong\pi_{1}Z_{t-1}\Omega^{k}\Xd}
and so on. This does not imply, of course, that 
\w[.]{\pinatt{t,1}{k}\Xd\cong\pinat{t,k+1}\Xd}

We therefore see from \wref{eqrowcol} and \wref{eqpartial} that the
differential  \w{d^{2}_{t,k}:E^{2}_{t,k}\to E^{2}_{t-2,k+1}} may be
identified with: 
\begin{myeq}[\label{eqdtwodiff}]
\pi_{t}\pi_{k}\Xd \cong \pi_{t}\pi_{0}\q{k,1}\Xd 
\xra{\partial_{t,0}\q{k,1}} \Omega\pinatt{t-2,0}{k}\Xd=
\pinatt{t-2,1}{k}\Xd ~\stackrel{h_{t}}{\cong}~ 
\pi_{t-2}\pi_{1}\q{k,1}\Xd\cong\pi_{t-2}\pi_{k+1}\Xd
\end{myeq}

Now by definition, \w{E^{3}_{t,k}} fits into a commutative diagram:
\mydiagram[\label{eqetterm}]{
E^{2}_{t+2,k-1} \ar[r]^<<<<<<<{d^{2}_{t+2,k-1}} \ar@{->>}[d]_{r} &
E^{2}_{t,k} \ar@{->>}[r]^<<<<<{q} & \Cok(d^{2}_{t+2,k-1}) \\
\Image(d^{2}_{t+2,k-1}) \ar@{^{(}->}[r]^{\ell} & 
\Ker(d^{2}_{t,k}) \ar@{^{(}->}[u]_{j} \ar@{->>}[r]^{s} & 
E^{3}_{t,k} \ar@{^{(}->}[u]_{\kappa}
}
\noindent with exact rows, $\ell$ $j$ and $\kappa$ monic, and thus 
\w[.]{E^{3}_{t,k}\cong\Image(q\circ j)}

From the exactness of \wref{eqpspiral} (together with \wref[)]{eqbottom} 
we see that 
\w{\Cok(d^{2}_{t+2,k-1})=\Cok(\partial_{t+2}\q{k-1,1})=\Image(s_{t+1}\q{k-1,1})}
and 
\w[,]{\Ker(d^{2}_{t,k})=\Ker(\partial_{t}\q{k,1})=\Image(h_{t}\q{k,1})}
so \w{E^{3}_{t,k}=\Image(q\circ j)} is indeed the image of the map
in \wref[.]{eqetht}
\end{proof}

\begin{defn}\label{dtss}
An $r$-\emph{truncated spectral sequence} is one defined up to and
including the \Et{r}, together with the differential
\w[,]{d^{n}:E^{r}_{t,i}\to E^{r}_{t-r-1,t+r}} but without
requiring that \w{d^{r}\circ d^{r}=0} (so the \Et{r+1}\ is defined in 
terms of the $r$-truncated spectral sequence only if \w[).]{d^{r}d^{r}=0}
\end{defn}

The main example is the $n$-truncation of an (ordinary)
spectral sequence (such as that of a simplicial space). In this case
we do have \w[,]{d^{r}\circ d^{r}=0} of course.

%
%
\begin{cor}\label{cett}
Any Reedy fibrant simplicial Postnikov $1$-stem has a well-defined
$2$-truncated spiral spectral sequence. Moreover, if \w{\Qd=\PP[1]\Xd}
for some simplicial space  \w[,]{\Xd} this $2$-truncated spectral
sequence coincides with the $2$-truncation of the
Bousfield-Friedlander spectral sequence for \w[.]{\Xd} 
\end{cor}

In general, we have a less explicit description of the higher terms in
the spiral spectral sequence:

%
%
\begin{thm}\label{tert}
For each \w[,]{r\geq 0} the \Et{r+2}\ of the spiral spectral sequence
for a simplicial space \w{\Xd} is determined by the spiral $r$-system
of \w[.]{\Xd} Moreover, for any \w[,]{\alpha\in E^{r+1}_{t,i}} we have 
\w{d^{r+1}_{t,i}(\alpha)=\beta\in E^{r+1}_{t-r-1,i+r}} if and only if
$\alpha$ and $\beta$ have representatives \w{\bar{a}\in\pi_{t}\pi_{i}\Xd}
and \w[,]{\bar{b}\in\pi_{t-r-1}\pi_{i+r}\Xd} respectively, such that:
\begin{myeq}[\label{eqdrpo}]
(s_{t-2,1}\q{i,r})\circ(s_{t-3,2}\q{i,r})\circ\dotsb\circ(s_{t-r,r-1}\q{i,r})
\circ(h_{t-r-1,r}\q{i,r})^{-1}(\bar{b})~=~\partial_{t,0}\q{i,r}(\bar{a})
\end{myeq}
\end{thm}

\begin{proof}
We naturally identify \w{\pinatt{t,k}{i,r}\Xd} 
with \w{\pinatt{t,k+s}{i,r-s}\Xd} for \w[,]{k\geq s} and similarly for
the maps in \wref[,]{eqpspiral} so the spiral \ww{(r-1)}-system embeds
in the spiral $r$-system (with an index shift).

Again we write out the \Et{1}\ of the spiral exact couple:
\begin{equation*}
\xymatrix@R=20pt{
& \pi_{r}C_{t-r}\Po{r}\Omega^{i}\Xd\hs \ar[r]^{(d_{0}^{t-r})_{\#}} & 
\hsm\pi_{r}Z_{t-r-1}\Po{r}\Omega^{i}\Xd \ar[r]^{(j_{t-r-1})_{\#}} 
\ar[rd]^{(j_{t-r-1}^{\Xd})_{\#}} \ar@{->>}[d]^{\hvart_{t-r-1}} & 
\pi_{r}C_{t-r-1}\Po{r}\Omega^{i}\Xd \\
&& \quad\hsnn \Omega\pinatt{t-r-1,r-1}{i,r}\Xd=\pinatt{t-r-1,r}{i,r}\Xd 
\ar@{^{(}->}[dd]^{\ell_{t-r-1,r}} 
\ar[rd]^>>>>>>>>>>>{h_{t-r-1,r}\q{i,r}}_{\cong} & 
Z_{t-r-1}\pi_{i+r}\Xd \ar@{^{(}->}[u]_{\inc} \ar@{->>}[d]^{\vart_{t-r-1}}\\ 
&&&\pi_{t-r-1}\pi_{i+r}\Xd\\
& \pi_{r-1}C_{t-r+1}\Po{r}\Omega^{i}\Xd\hs \ar[r]^{(d_{0}^{t-r+1})_{\#}} &
\pi_{r-1}Z_{t-r}\Po{r}\Omega^{i}\Xd \ar[r]^{(j_{t-r})_{\#}} &
\pi_{r-1}C_{t-r}\Po{r}\Omega^{i}\Xd \\
%
& \vdots & \vdots &\vdots& \\
& \pi_{2}C_{t-2}\Po{r}\Omega^{i}\Xd \ar[r]^{(d_{0}^{t-2})_{\#}} & 
\pi_{2}Z_{t-3}\Po{r}\Omega^{i}\Xd \ar[r]^{(j_{t-3})_{\#}} 
\ar[rd]^{(j_{t-3}^{\Xd})_{\#}} \ar@{->>}[d]^{\hvart_{t-3}} &
\pi_{2}C_{t-3}\Po{r}\Omega^{i}\Xd \\
&& \quad\hsnn \Omega\pinatt{t-3,1}{i,r}\Xd=\pinatt{t-3,2}{i,r}\Xd 
\ar@/_11pc/@<-13ex>[ddd]_{s_{t-3,1}\q{i,r}}
\ar@{^{(}->}[dd]^{\ell_{t-3,2}} \ar[rd]^>>>>>>>>>>>{h_{t-3,2}\q{i,r}} & 
Z_{t-3}\pi_{i+2}\Xd \ar@{^{(}->}[u]_{\inc} \ar@{->>}[d]^{\vart_{t-3}}\\ 
&&&\pi_{t-3}\pi_{i+2}\Xd\\
& \pi_{1}C_{t-1}\Po{r}\Omega^{i}\Xd \ar[r]^{(d_{0}^{t-1})_{\#}} & 
\pi_{1}Z_{t-2}\Po{r}\Omega^{i}\Xd \ar[r]^{(j_{t-2})_{\#}} 
\ar[rd]^{(j_{t-2}^{\Xd})_{\#}} \ar@{->>}[d]^{\hvart_{t-2}} & 
\pi_{1}C_{t-2}\Po{r}\Omega^{i}\Xd \\
&& \quad\hsnn \Omega\pinatt{t-2,0}{i,r}\Xd=\pinatt{t-2,1}{i,r}\Xd 
\ar@{^{(}->}[dd]^{\ell_{t-2,1}} \ar[rd]^>>>>>>>>>>>{h_{t-2,1}\q{i,r}} & 
Z_{t-2}\pi_{i+1}\Xd \ar@{^{(}->}[u]_{\inc} \ar@{->>}[d]^{\vart_{t-2}}\\ 
&&&\pi_{t-2}\pi_{i+1}\Xd\\
\pi_{0}Z_{t}\Po{r}\Omega^{i}\Xd \ar[r]^{(j_{t})_{\#}} 
\ar[rd]^{(j_{t}^{\Xd})_{\#}}  \ar@{->>}[d]^{\hvart_{t}} & 
\pi_{0}C_{t}\Po{r}\Omega^{i}\Xd \ar[r]^<<<<{(d_{0}^{t})_{\#}} & 
\pi_{0}Z_{t-1}\Po{r}\Omega^{i}\Xd \ar[r]^{(j_{t-1})_{\#}} & 
\pi_{0}C_{t-1}\Po{r}\Omega^{i}\Xd \to \dotsc\\
\pinatt{t,0}{i}\Xd \ar[rd]^{h_{t}\q{i,r}} & 
Z_{t}\pi_{i}\Xd \ar@{^{(}->}[u]_{\inc} \ar@{->>}[d]_{\vart_{t}} & & \\
&\pi_{t}\pi_{i}\Xd \ar[ruuuu]^>>>>>>>>>>>>>{\partial_{t,0}\q{i,r}}&&
}
\end{equation*}
\noindent The differential 
\w{d^{r+1}_{t,i}:E^{r+1}_{t,i}\to E^{r+1}_{t-r-1,i+r}} may then be
described as a ``relation'' (cf.\ \cite[\S 3.1]{BKanSQ}) in the usual way:

Given a class \w[,]{\alpha\in E^{r+1}_{t,i}} choose a representative
for it \w[.]{a\in E^{1}_{t,i}=\pi_{0}C_{t}\Po{r}\Omega^{i}\Xd} Since it
is a cycle for \w[,]{d^{1}_{t,i}=(j_{t-1})_{\#}\circ(d_{0}^{t})_{\#}}
it lies in \w{Z_{t}\pi_{i}\Xd} and thus represents an element 
\w[.]{\bar{a}\in\pi_{t}\pi_{i}\Xd=E^{2}_{t,i}} From the exactness of
the middle row of \wref{eqrowcol} we see that 
\w[,]{(d_{0}^{t})_{\#}(a)\in\Ker((j_{t-1})_{\#})=\Omega\pinatt{t-2,0}{i,r}\Xd}
and in fact  \w{(d_{0}^{t})_{\#}(a)} represents 
\w[.]{\partial_{t,0}\q{i,r}(\bar{a})} Since \w{\hvart_{t-2}} is
surjective, we can choose
\w{e_{t-2}\in\pi_{1}Z_{t-2}\Po{r}\Omega^{i}\Xd} mapping to
\w[.]{(d_{0}^{t})_{\#}(a)} Because  
\w[,]{d^{2}_{t,i}(\bar{a})=h_{t-2,1}\q{i,r}\circ\partial_{t,0}\q{i,r}(\bar{a})}
as in the proof of Proposition \ref{pett} (though \w{h_{t-2,1}\q{i,r}}
need no longer be an isomorphism!), we see that it is represented by 
\w[.]{(j_{t-2})_{\ast}(e_{t-2})} If \w[,]{r=1} we are done. Otherwise, 
we know that \w[,]{d^{2}_{t,i}(\bar{a})=0} so we can choose
\w{e_{t-2}} so that \w[,]{(j_{t-2})_{\ast}(e_{t-2})=0} using exactness
of the third column of of \wref[.]{eqrowcol} Again this implies that 
\w[,]{e_{t-2}\in\Ker((j_{t-2})_{\#})=\Omega\pinatt{t-3,1}{i,r}\Xd} and 
\w{d^{3}_{t,i}(\lra{a})} is represented by
\w[.]{h_{t-3,2}\q{i,r}(e_{t-2})} Moreover, we see from
\wref{eqrowcol} that
\w[,]{s_{t-3,1}\q{i,r}(e_{t-2})=\partial_{t,0}\q{i,r}(\bar{a})}
using the identification
\w[.]{\Omega\pinatt{t-2,0}{i,r}\Xd=\pinatt{t-2,1}{i,r}\Xd } 

Choosing a lift to
\w[,]{e_{t-3}\in\pi_{2}Z_{t-3}\Po{r}\Omega^{i}\Xd} we may assume that
\w[,]{(j_{t-3})_{\ast}(e_{t-3})=0} so 
\w{e_{t-3}\in\Omega\pinatt{t-4,2}{i,r}\Xd} and
\w[.]{s_{t-4,2}\q{i,r}(e_{t-3})=e_{t-2}} 
Continuing in this way, we finally reach
\w{e_{t-r-1}\in\Omega\pinatt{t-r-1,r-1}{i,r}\Xd} with 
\w[,]{s_{t-r-2,r}\q{i,r}(e_{t-r-1})=e_{t-r}} and so on,
and see that \w{d^{r+1}_{t,i}(\lra{a})} is represented by
\w[.]{h_{t-r-1,r}\q{i,r}(e_{t-r-1})} Since (as in the proof of
Proposition \ref{pett}) \w{h_{t-r-1,r}\q{i,r}} is an isomorphism,
we deduce that \w{d^{r+1}_{t,i}(\alpha)} is as in \wref[.]{eqdrpo}
\end{proof}

\begin{remark}\label{rert}
From the exactness of \wref{eqpspiral} we have 
\w[,]{\Image(\partial_{t,0}\q{i,r})=\Ker(s_{t-1,0}\q{i,r})} so the
image of \w{d_{t,i}^{r+1}} as described in \wref{eqdrpo} is 
\w[,]{\Ker(\sigma_{t,i}^{r+1})} where
\w[.]{\sigma_{t,i}^{r+1}:=(s_{t-1,0}\q{i,r})\circ(s_{t-2,1}\q{i,r})
\circ(s_{t-3,2}\q{i,r})\circ\dotsb\circ(s_{t-r,r-1}\q{i,r})}
Therefore, \w{E^{r+1}_{t+r-1,i+r}} embeds naturally in 
\w[.]{\Image(\sigma_{t,i}^{r+1})}
\end{remark}

%
%
\begin{cor}\label{cert}
Every Reedy fibrant simplicial Postnikov $r$-stem has a well-defined
\ww{(r+1)}-truncated spiral spectral sequence. If
\w{\Qd=\PP[r]\Xd} for some simplicial space \w[,]{\Xd} this truncated
spectral sequence coincides with the \ww{(r+1)}-truncation of the
Bousfield-Friedlander spectral sequence for \w[.]{\Xd}  
\end{cor}

Thus the bigraded homomorophism 
$$
d^{r+1}\circ d^{r+1}:E^{r}_{t,i}\to E^{r+1}_{t-2r-2,i+2r}\hsp 
(t\geq 2r+2, i\geq 0)
$$
\noindent serves as the first obstruction to the realizablity of the
simplicial Postnikov $r$-stem \w{\Qd} by a simplicial space \w[.]{\Xd}

%
%
\sect{A cosimplicial version}
\label{ccsv}

There are actually four variants of the above spectral sequence which
we might consider, for a simplicial or cosimplicial object over
simplicial or cosimplicial sets.  The case of bicosimplicial sets is
in principle strictly dual to that of bisimplicial sets, but because
the category of cosimplicial \emph{sets} has no (known) useful model
category structure, we must restrict to bicosimplicial abelian groups
\wh or equivalently, ordinary double complexes. Thus the main new case
of interest is that of cosimplicial simplicial sets, or
\emph{cosimplicial spaces}.

\begin{mysubsection}[\label{ssscs}]{The spectral sequence of a cosimplicial space}
If \w{\Xu\in c\Sa} is a fibrant cosimplicial pointed space with total
space \w[,]{\Tot\Xu} there are various constructions for the homotopy
spectral sequence of \w[:]{\Xu} 

\begin{enumerate}
\renewcommand{\labelenumi}{(\alph{enumi})}
\item Using the tower of fibrations for \w{(\Tot_{n}\Xu)_{n=0}^{\infty}} 
(cf.\ \cite[X,\S 6]{BKanH}). 
\item Using ``relations'' on the normalized cochains
\w{N^{n}\pi_{t}\Xu:=\pi_{t}X^{n}\cap\Ker(s^{0})\cap\dotsc\cap\Ker(s^{n-1})}
(cf.\ \cite[\S 7]{BKanSQ}).
\item Using a cofibration sequence dualizing \wref{eqfibseq} 
(cf.\ \cite[\S 3]{RectS}).
\end{enumerate}

Bousfield and Kan showed that the result is essentially unique 
(see \cite{BKanSQ}). Since the main ingredient needed for to define
the spiral exact couple is the diagram \wref[,]{eqrowcol} we use the
first approach:
\end{mysubsection}

\begin{defn}\label{dncgps}
For any Reedy fibrant cosimplicial pointed space \w[,]{\Xu\in c\Sa} 
consider the fibration sequence
\begin{myeq}[\label{eqfibtot}]
F_{n}\Xu~\xra{j_{n}}~\Tot_{n}\Xu~\xra{p_{n}}~\Tot_{n-1}\Xu~,
\end{myeq}
\noindent where \w{\Tot_{n}\Xu:=\map_{c\Sa}(\sk{n}\bDel,\Xu)} and the
fibration \w{p_{n}} is induced by the inclusion of cosimplicial spaces 
\w[.]{\sk{n-1}\bDel\hra\sk{n}\bDel} 

The cokernel of \w{(j_{n})_{\#}:\pis F_{n}\Xu\hra\pis\Tot_{n}\Xu} is
called  the $n$-th \emph{natural (graded) cohomotopy group} of \w[,]{\Xu}
and denoted by \w[.]{\punat{\ast}{n}\Xu}
\end{defn}

\begin{remark}\label{rncgps}
We may identify \w{F_{n}\Xu} with the looped normalized cochain object
\w[,]{\Omega^{n}N^{n}\Xu} where
\begin{myeq}[\label{eqnch}]
N^{n}\Xu~:=~X^{n}\cap\Ker(s^{0})\cap\dotsc\cap\Ker(s^{n-1})~,
\end{myeq}
\noindent and \w{\pis N^{n}\Xu} with \w{N^{n}\pis\Xu} (see 
\cite[X, Proposition 6.3]{BKanH}). 

Moreover, the composite
$$
\pisp\Omega^{n}N^{n}\Xu\cong\pisp F_{n}\Xu~\xra{(j_{n})_{\#}}~
\pisp\Tot_{n}\Xu~\xra{\partial_{n}}~
\pis F_{n+1}\Xu\cong\pis\Omega^{n+1}N^{n+1}\Xu
$$
\noindent \noindent (where \w{\partial_{n}} is the connecting
homomotphism for the \wref[),]{eqfibtot} may then be identified with
the differential
\begin{myeq}[\label{eqdelta}]
\delta^{n}:=\sum_{i=0}^{n}(-1)^{i}d^{i}~:N^{n}\pis\Xu\to N^{n+1}\pis\Xu~,
\end{myeq}
\noindent for the normalized cochain complex \w[,]{N^{\ast}\pis\Xu} so that 
\begin{myeq}[\label{eqcett}]
\Ker(\delta^{n})/\Cok(\delta^{n+1})\cong\pi^{n}\pis\Xu
\end{myeq}
\noindent (cf.\ \cite[X, \S 7.2]{BKanH}). 
\end{remark}

%
%
\begin{prop}\label{pcsles}
For any pointed cosimplicial space \w{\Xu} there is a natural
\emph{spiral} long exact sequence:
\begin{myeq}[\label{eqcspiral}]
\begin{split}
\ldots~\to~\Omega\punat{\ast}{n-1}\Xu&~\xra{s^{n}}~\punat{\ast}{n}\Xu~\xra{h^{n}}~
\pi^{n}\pis\Xu~\xra{\partial^{n}}~\Omega\punat{\ast}{n-2}\Xu\\ 
&~\xra{s^{n-1}}~\punat{\ast}{n-1}\Xu~\to~\ldots~\to~
\punat{\ast}{0}\Xu~\xra{\cong}~\pi^{0}\pis\Xu 
\end{split}
\end{myeq}
\end{prop}

\begin{proof}
By choosing a fibrant replacement in the model category of
cosimplicial simplicial sets defined in \cite[X, \S 5]{BKanH}, if
necessary, we may assume that \w{\Xu} is Reedy fibrant. 
We then obtain a commutative diagram as in \wref{eqrowcol} with exact
rows and columns:  
\mydiagram[\label{eqcrowcol}]{
& 0 \ar[d] & 0 \ar[d] & 0 \ar[d] & & \\ 
0 \ar[r] & \Ker(j_{n})_{\ast}~ \ar@{^{(}->}[r] \ar@{^{(}->}[d] & 
B^{n+1}\Xu \ar@{^{(}->}[d]
\ar@{->>}[r]^>>>>>{(j_{n})_{\ast}} & B^{n+1}\pis\Xu\ar[r] \ar@{^{(}->}[d] & 
0\ar[d] & \\
0 \ar[r] & \Omega\punat{\ast}{n-1}\Xu~ \ar@{^{(}->}[r]^{\ell_{n-1}} 
\ar@{->>}[d] \ar[rd]^{s_{n}} & 
\pis \Tot_{n}\Xu \ar[d]^{\hvart_{n}} \ar[r]^{(j_{n}^{\Xu})_{\#}} & 
Z^{n}\pis\Xu \ar@{->>}[r] \ar[d]^{\vart_{n}} & \Cok h^{n} \ar[r] \ar[d]^{=} & 0 \\
0 \ar[r] & \Ker h^{n}~ \ar@{^{(}->}[r] \dto & 
\punat{\ast}{n}\Xu \ar[d] \ar[r]^{h_{n}} & 
\pi^{n}\pis\Xu \ar@{->>}[r] \ar[d] & \Cok h^{n}\ar[r] \ar[d] & 0 \\
 & 0 & 0 & 0 & 0 &
}
\noindent in which 
\w{B^{n+1}\Xu:=\Image(j_{n+1})_{\#}\subseteq\pis\Tot_{n}\Xu} and 
\w{B^{n+1}\pis\Xu:=\Image(\delta^{n+1})=
\Image(\partial_{n+1}\circ(j_{n+1})_{\#})} 
are the respective coboundary objects. 

The construction of the maps \w[,]{h^{n}} \w[,]{s^{n}}and \w[,]{\partial^{n}}
and the proof of the exactness of \wref[,]{eqcspiral} are then precisely
as in \S \ref{scsles}. 
\end{proof}

\begin{defn}\label{dcrsys}
The \emph{spiral $n$-system} of a pointed cosimplicial space 
\w{\Xu\in c\Sa} is defined to be the collection of long exact 
sequences \wref{eqcspiral} for the Postnikov $n$-stem functor
\w{\PP[n]} applied to \w[,]{\Xu} one for each $k$-window of
\w[.]{\PP[n]\Xu} 

As in Definition \ref{dsns}, this may actually be defined for a
cosimplicial Postnikov $n$-stem \w[,]{\Pu} not necessarily realizable
as \w[.]{\Pu=\PP[n]\Xu}
\end{defn}

By construction, the homotopy spectral sequence of a (fibrant)
cosimplicial space \w[,]{\Xu} obtained as in \wref[,]{ssscs} is
associated to the spiral exact couple \wref[.]{eqcspiral}
The proofs of Proposition \ref{pett} and Theorem \ref{tert} use only
the description of the spiral exact couple for \w{\Xd} derived from
\wref[,]{eqcrowcol} so by using \wref{eqcrowcol} instead we can prove
their analogues in the cosimplicial case, and show:

%
%
\begin{thm}\label{tcert}
The $E_{r+2}$-term of the homotopy spectral sequence for a cosimplicial
space \w{\Xu} is determined by the spiral $r$-system of \w[.]{\Xu} 
\end{thm}

An analogue of Corollary \ref{cert} also holds, as well as:

%
%
\begin{prop}\label{pcett}
The differential \w{d_{2}^{t,i}:E_{2}^{t,i}\to E_{2}^{t+2,i+1}} may be 
identified with 
\w[.]{\partial^{t}\li{i,1}:\pi^{t}\pi_{i}\Xu\to\Omega\punatt{t+2,0}{i}\Xu}
\end{prop}

\begin{examples}\label{egert}
As noted in the introduction, many commonly used spectral sequences
arise as the spiral spectral sequence of an appropriate
(co)simplicial space, so Theorems \ref{tert} and \ref{tcert} allow us
to extract their $E^{r}$- or $E_{r}$-terms from the appropriate spiral
systems. For instance:
 
\begin{enumerate}
\renewcommand{\labelenumi}{(\alph{enumi})}
\item Segal's homology spectral sequence (cf.\ \cite{SegC}), the van Kampen
  spectral sequence (cf.\ \cite{StoV}), and the Hurewicz spectral
  sequence (cf.\ \cite{BlaH}) are constructed using bisimplicial sets. 
\item The unstable Adams spectral sequences of \cite{BCKQRSclM,BKanS}
  and \cite[\S 4]{BCMillU}, Rector's version of the Eilenberg-Moore
  spectral sequence (cf.\ \cite{RectS}), and Anderson's generalization
  of the latter  (cf.\ \cite{AndG}) are all associated to cosimplicial spaces.
\item The usual construction of the stable Adams spectral sequence for
  \w{\pis^{s}X\otimes\bZ/p} (cf.\ \cite[\S 3]{AdSS}) uses a tower of
  (co)fibrations, rather than a cosimplicial space, but when $X$ is 
  finite dimensional, it agrees in a range with the unstable version
  for \w[,]{\Sigma^{N}X} so Theorem \ref{tcert} applies stably, too. 
\end{enumerate}
\end{examples}


\begin{thebibliography}{BCKQRS}
%
\bibitem[Ad]{AdSS}
J.F.~Adams,
``On the structure and applications of the Steenrod algebra'',\hsm
\textit{Comm.\ Math.\ Helv.} \textbf{32} (1958), pp.~180-214.
%
\bibitem[An]{AndG}
D.W.~Anderson,
``A generalization of the Eilenberg-Moore spectral sequence'',\hsm
\textit{Bull.\ AMS} \textbf{78} (1972), pp.~784-786.
%
\bibitem[Ba]{BauHH}
H.-J.~Baues,
\textit{Homotopy type and homology},\hsm
Oxford U.\ Press, Oxford-\-New York, 1996.
%
\bibitem[BB1]{BBlaC}
H.-J.~Baues \& D.~Blanc,
``Comparing cohomology obstructions'',\hsm
\textit{J.\ Pure \& Appl.\ Alg.}, to appear.
%
\bibitem[BB2]{BBlaH}
H.-J.~Baues \& D.~Blanc,
``Higher order derived functors and spectral sequences'',\hsm
preprint, 2010.
%
\bibitem[BJ]{BJiblSD}
H.-J.~Baues \& M.A.~Jibladze,
``Secondary derived functors and the {Adams} spectral sequence'',\hsm
\textit{Topology} \textbf{45} (2006), pp.~295-324.
%
\bibitem[BM]{BMuroS}
H.-J.~Baues \& F.~Muro,
``Secondary homotopy groups'',\hsm
\textit{Forum Mathematicum} \textbf{20} (2008), pp.~631-677.
%
\bibitem[BCM]{BCMillU}
M.~Bendersky, E.B.~Curtis, \& H.R.~Miller,
``The unstable Adams spectral sequence for generalized homology'',\hsm
\textit{Topology} \textbf{17} (1978), pp.~229-148.
%
\bibitem[Bl1]{BlaH}
D.~Blanc,
``A Hurewicz spectral sequence for homology'',\hsm
\textit{Trans.\ AMS} \textbf{318} (1990), pp.~335-354.
%
\bibitem[Bl3]{BlaCW}
D.~Blanc,
``CW simplicial resolutions of spaces, with an application to loop spaces'',\hsm 
\textit{Top.\ \& Appl.} \textbf{100} (2000), pp.~151-175.
%
\bibitem[Bl4]{BlaC}
D.~Blanc,
``Comparing homotopy categories'',\hsm 
\textit{J.\ $K$-Theory} \textbf{2} (2008), pp.~169-205.
%
\bibitem[Bou1]{BousH}
A.K.~Bousfield,
``Homotopy Spectral Sequences and Obstructions'',\hsm
\textit{Israel J.\ Math.} \textbf{66} (1989), pp.~54-104.
%
\bibitem[Bou2]{BousCR}
A.K.~Bousfield,
``Cosimplicial resolutions and homotopy spectral sequences in 
model categories'',\hsm
\textit{Geom.\ \& Topology} \textbf{7} (2003), pp.~1001-1053.
%
\bibitem[BCKQRS]{BCKQRSclM}
A.K.~Bousfield, E.B.~Curtis, D.M.~Kan, D.G.~Quillen, D.L.~Rector, \& 
J.W.~Schlesinger,
``The mod-$p$ lower central series and the Adams spectral sequence'',\hsm
\textit{Topology} \textbf{3} (1966), pp.~331-342.
%
\bibitem[BF]{BFrieH} 
A.K.~Bousfield \& E.M.~Friedlander, 
``Homotopy theory of $\Gamma$-spaces, spectra, and bisimplicial sets'',\hsm 
in M.G.~Barratt \& M.E.~Mahowald, eds., 
\textit{Geometric Applications of Homotopy Theory, II} 
Springer \textit{Lec.\ Notes Math.} \textbf{658}, Berlin-\-New York, 1978,
pp.~80-130.
%
\bibitem[BK1]{BKanH}
A.K.~Bousfield \& D.M.~Kan,
\textit{Homotopy Limits, Completions, and Localizations},\hsm
Springer \textit{Lec.\ Notes Math.} \textbf{304}, Berlin-\-New York, 1972.
%
\bibitem[BK2]{BKanS}
A.K.~Bousfield \& D.M.~Kan,
``The homotopy spectral sequence of a space with coefficients in a ring'',\hsm
\textit{Topology} \textbf{11} (1972), pp.~79-106.
%
\bibitem[BK3]{BKanSQ}
A.K.~Bousfield \& D.M.~Kan,
``A second quadrant homotopy spectral sequence'',\hsm
\textit{Trans.\ AMS} \textbf{177} (1973), pp.~305-318.
%
\bibitem[DKSt1]{DKStE} 
W.G.~Dwyer, D.M.~Kan, \& C.R.~Stover, 
``An $E^{2}$ model category structure for pointed simplicial spaces'',\hsm 
\textit{J.\ Pure \& Appl.\ Alg.} \textbf{90} (1993), pp.~137-152.
%
\bibitem[DKSt2]{DKStB} 
W.G.~Dwyer, D.M.~Kan, \& C.R.~Stover, 
``The bigraded homotopy groups $\pi_{i,j}X$ of a pointed simplicial 
space'',\hsm 
\textit{J.\ Pure Appl.\ Alg.} \textbf{103} (1995), pp.~167-188.
%
\bibitem[Hi]{PHirM}
P.S.~Hirschhorn,
\textit{Model Categories and their Localizations},\hsm
Math.\ Surveys \& Monographs \textbf{99}, AMS, Providence, RI, 2002.
%
\bibitem[Q]{QuiS} 
D.G.~Quillen,
``Spectral sequences of a double semi-\-simplicial group'',\hsm 
\textit{Topology} \textbf{5} (1966), pp.~155-156.
%
\bibitem[R]{RectS}
D.L.~Rector,
``Steenrod operations in the Eilenberg-Moore spectral sequence'',\hsm
\textit{Comm.\ Math.\ Helv.} \textbf{45} (1970), pp.~540-552.
%
\bibitem[Se]{SegC} 
G.~Segal, 
``Classifying spaces and spectral sequences'',\hsm
\textit{Publ.\ Math.\ Inst.\ Haut.\ \'{E}t.\ Sci.} \textbf{34} (1968), 
pp.~105-112.
%
\bibitem[St]{StoV}
C.R.~Stover,
``A Van Kampen spectral sequence for higher homotopy groups'',\hsm
\textit{Topology} \textbf{29} (1990), pp.~9-26.
%
\end{thebibliography}
\end{document}